\documentclass[11pt,a4paper]{article}

\usepackage[T1]{fontenc}
\usepackage[utf8]{inputenc}

\usepackage{amsmath,amssymb,amsthm,mathtools}
\usepackage{geometry}
\usepackage{hyperref}
\usepackage{url}
\usepackage{booktabs}

\newtheorem{proposition}{Proposition}
\theoremstyle{definition}
\newtheorem{definition}{Definition}
\newtheorem{remark}{Remark}

\newcommand{\ES}{\mathrm{ES}}

\title{Notes on the 33-point Erd\H{o}s--Szekeres problem}
\author{Dumitru Bogdan
\thanks{Faculty of Computer Science, University of Bucharest, bogdan.dumitru@fmi.unibuc.ro}}
\date{December, 2025}

\begin{document}
\maketitle

\begin{abstract}
The determination of $\ES(7)$ is the first open case of the planar Erd\H{o}s--Szekeres problem, where the general conjecture predicts $\ES(7)=33$.
We present a SAT encoding for the 33-point case based on triple-orientation variables and a 4-set convexity criterion for excluding convex $7$-gons, together with convex-layer anchoring constraints.
The framework yields UNSAT certificates for a collection of anchored subfamilies. We also report pronounced runtime variability across configurations, including heavy-tailed behavior that currently dominates the computational effort and motivates further encoding refinements.
\end{abstract}


\section{Introduction}

In their seminal 1935 paper, Erd\H{o}s and Szekeres investigated the smallest integer $\ES(k)$ such that any set of $\ES(k)$ points in the plane in general position (no three collinear) contains $k$ points in convex position~\cite{Erdos1935}. They conjectured the exact formula
\[
\ES(k) \;=\; 2^{k-2}+1 \qquad (k\ge 3).
\]
This formula is verified for $k\le 6$. Moreover, the classical Erd\H{o}s--Szekeres construction gives the matching lower bound $\ES(k)\ge 2^{k-2}+1$~\cite{Erdos1935}. The case $k=7$ remains open: it is not known whether every set of $33$ points in general position contains a convex $7$-gon, i.e., whether $\ES(7)=33$.

Upper bounds for $\ES(k)$ have been refined substantially; for instance Suk proved $\ES(k)\le 2^{k+o(k)}$~\cite{Suk16}, and Holmsen--Mojarrad--Pach--Tardos improved this to $\ES(k)\le 2^{k+O(\sqrt{k\log k})}$~\cite{holmsen2020extensions}. Computational approaches have also played an important role. Szekeres and Peters gave a computer-assisted proof that $\ES(6)=17$~\cite{Szekeres2006}, and SAT-based approaches have been explored in related settings (e.g.\ \cite{Balko2017,Mari2017}). Recent years have seen renewed SAT activity on Erd\H{o}s--Szekeres-type problems: Scheucher developed SAT models for higher-dimensional Erd\H{o}s--Szekeres parameters and related questions, verifying UNSAT results via proof certificates~\cite{Scheucher2023}, and Heule and Scheucher established the exact empty hexagon number $h(6)=30$ using a compact SAT encoding together with large-scale search and proof production~\cite{HeuleScheucher2024}. While empty-hole problems and higher-dimensional variants differ from the classical (non-empty) planar $\ES(7)$ case, these works illustrate the current reach of SAT-based methods and motivate further exploration of encodings tailored to the 33-point problem.

In this work we investigate the $n=33$ case computationally by encoding geometric consistency and the exclusion of convex $7$-gons as a Boolean satisfiability problem. Our encoding is built around orientation variables on triples and an exact reduction of convex position to constraints on 4-point subsets. To reduce symmetry and support systematic experimentation, we additionally impose convex-layer (hull-template) anchoring constraints and, for some hard templates, a simple sub-cubing parameter that further fixes the relative alignment of consecutive layers. Using this framework we obtain UNSAT certificates for a growing collection of anchored subfamilies, and we report a pronounced runtime imbalance across configurations, with some subproblems taking weeks on commodity hardware.

The paper is organized as follows: Section~2 introduces the triple-orientation variables and a family of 5-point CC-style implications (after Knuth).
Section~3 states the 4-set criterion for convex position and shows how it yields a compact ``no convex 7-set'' constraint.
Section~4 describes convex-layer anchoring templates and a simple sub-cubing parameter.
Sections~5--6 summarize instance sizes and representative computational results.
Section~7 discusses bottlenecks and outlines future work.

\section{Preliminaries: triple orientations and CC-style constraints}

\subsection{Orientation variables}

Let $[n]=\{0,1,\dots,n-1\}$ be labels for points. In a realizable point set in general position, every ordered triple $(a,b,c)$ has a well-defined orientation (clockwise vs.\ counterclockwise). We represent this with a sign predicate
\[
\chi(a,b,c)\in\{+,-\},
\]
alternating under swapping arguments.

\begin{definition}[3-cup / 3-cap (convention)]
Fix a convention: $\chi(a,b,c)=+$ corresponds to a \textbf{3-cup} and $\chi(a,b,c)=-$ corresponds to a \textbf{3-cap} (or vice versa). The convention is arbitrary but must be used consistently.
\end{definition}

We introduce one Boolean variable for each \emph{unordered} triple $\{i,j,k\}$ with $i<j<k$:
\[
x_{ijk}\in\{0,1\}.
\]

For any ordered triple $(a,b,c)$, let $(i,j,k)$ be the sorted triple and define a signed literal
\[
\ell(a,b,c)\in\{\pm x_{ijk}\}
\]
depending on whether $(a,b,c)$ is an even or odd permutation of $(i,j,k)$. This compactly enforces antisymmetry at the literal level.

\subsection{A 5-point implication (CC axiom 5)}

To restrict assignments toward realizable order types, one can enforce axioms of oriented matroids / chirotopes. We use a family of 5-point implications in the spirit of Knuth's CC axioms~\cite{Knuth1992-ae}. In one common form, for distinct points $p_1,\dots,p_5$:
\[
\chi(p_1,p_2,p_3)=\chi(p_1,p_2,p_4)=\chi(p_1,p_2,p_5)=\chi(p_1,p_3,p_4)=\chi(p_1,p_4,p_5)=+
\;\Longrightarrow\;
\chi(p_1,p_3,p_5)=+.
\]
Translated to CNF, this becomes a clause of the form
\[
\neg A_1 \vee \neg A_2 \vee \neg A_3 \vee \neg A_4 \vee \neg A_5 \vee C,
\]
where each $A_i$ and $C$ are signed literals of triple-orientation variables.

\paragraph{Reduced generation.}
In the implementation documented here, we generate a reduced subset of these 5-point clauses (motivated by symmetry and redundancy). This is a \emph{relaxation} used for performance.

\paragraph{Interpretation of solver outcomes.}
Omitting clauses can only enlarge the set of admissible Boolean assignments. Consequently, an \textbf{UNSAT} result for the relaxed instance remains a valid certificate for any stronger formulation that includes the omitted clauses.
Throughout this paper we report only UNSAT outcomes.

\section{A 4-set criterion for convex position}

A direct encoding of ``no convex $7$-gon'' tends to quantify over all 7-subsets and many cyclic orderings, creating a large clause blowup at $(n,k)=(33,7)$. We instead use an exact reduction to 4-point subsets.

\begin{proposition}[4-set criterion for convex position]\label{prop:4criterion}
Let $S$ be a finite set of points in the plane in general position.
Then $S$ is in convex position if and only if every 4-point subset of $S$ is in convex position.
\end{proposition}

\begin{proof}[Proof sketch]
If $S$ is in convex position, then every subset of $S$ is also in convex position, so every 4-point subset is convex.

Conversely, assume $S$ is \emph{not} in convex position. Then some point $p\in S$ is not a vertex of the convex hull of $S$, hence $p$ lies strictly inside the convex polygon $P=\mathrm{conv}(V)$, where $V\subseteq S$ is the set of hull vertices.
Triangulate $P$ by fixing a hull vertex $v_0\in V$ and drawing diagonals from $v_0$ to all non-adjacent hull vertices; this partitions $P$ into triangles whose vertices lie in $V$.
Since $p$ lies inside $P$, it lies inside one of these triangles, say $\triangle abc$ with $a,b,c\in V\subseteq S$.
Then the 4-set $\{a,b,c,p\}$ is not in convex position (one point lies inside the triangle formed by the other three), contradicting the assumption that every 4-point subset of $S$ is convex.
\end{proof}

\section{Encoding 4-point types and excluding convex 7-sets}

\subsection{Four cyclic triple orientations}

Fix a 4-set $\{a,b,c,d\}$ together with an ordering $(a,b,c,d)$. Consider the cyclic triples
\[
(a,b,c),\quad (b,c,d),\quad (c,d,a),\quad (d,a,b).
\]
Each triple has an orientation literal, hence the 4-set induces a length-4 sign pattern in $\{+,-\}^4$.

Not all $2^4=16$ patterns occur for 4 points in general position; exactly 14 do. In our current implementation, we introduce 14 selector variables per 4-set, one per realizable pattern, and constrain them to be mutually covering.

\subsection{Selector variables and CNF}

Formally, for each 4-set $Q$ and each realizable pattern index $p\in\{1,\dots,14\}$ we introduce a variable $t_{Q,p}$; we suppress the $Q$ subscript when discussing a fixed 4-set.

For each 4-set we introduce 14 selector variables $t_1,\dots,t_{14}$. Each selector is constrained to be equivalent to a conjunction of four literals (the cyclic triple orientations), using the standard reification template
\[
t \leftrightarrow (L_1\wedge L_2\wedge L_3\wedge L_4),
\]
encoded as 5 clauses. We also enforce that at least one selector holds. Thus each 4-set contributes $14\times 5+1 = 71$ clauses.

\begin{remark}[Convex vs.\ non-convex patterns]
In the generator version used here, 6 realizable patterns are treated as convex:
\[
++++,\;\; ----,\;\; ++--,\;\; --++,\;\; -++-,\;\; +--+,
\]
and the remaining 8 realizable patterns are treated as non-convex.
\end{remark}

\subsection{No convex 7-set clauses}

For each 7-set $K$, Proposition~\ref{prop:4criterion} implies that $K$ is in convex position if and only if all its 4-subsets are convex. Therefore, to exclude convex 7-sets it suffices to enforce that every 7-set contains at least one non-convex 4-subset. In our encoding, for each 7-set $K$ we add one clause
\[
\bigvee_{Q\in \binom{K}{4}}\;\bigvee_{p\in \mathcal{N}} t_{Q,p},
\]
where $\mathcal{N}$ indexes the non-convex patterns for a 4-set $Q$. This clause has length $35\times 8=280$.

\section{Convex layers (hull templates) and a simple sub-cubing parameter}

To reduce symmetry and to explore structured subfamilies, we add \emph{convex layer} constraints (also called convex-layer decompositions) that enforce a nested sequence of convex layers with prescribed sizes.

\subsection{Convex-layer (hull template) anchoring}

Fix $n$ and a vector of layer sizes
\[
\mathbf{h} = (h_0,h_1,\dots,h_{r-1}),\qquad h_i\ge 3,\qquad H\coloneqq \sum_{i=0}^{r-1} h_i \le n.
\]
We interpret points $0,\dots,h_0-1$ as the vertices of an outer convex $h_0$-gon in cyclic order, points $h_0,\dots,h_0+h_1-1$ as the next layer, and so on. Remaining $n-H$ points are treated as unconstrained interior points.

We then add unit clauses enforcing:
\begin{itemize}
\item \textbf{Within-layer convexity:} every triple of vertices within a layer (in the chosen cyclic order) has fixed orientation.
\item \textbf{Nesting:} each deeper-layer point lies on the interior side of each oriented edge of an outer layer (encoded by fixed orientations of suitable triples).
\end{itemize}

\subsection{A simple sub-cubing parameter (anchoring consecutive layers)}\label{subsec:subcubing}

For some hard convex-layer templates we further add a small family of unit clauses controlled by a vector
\[
\mathbf{w}=(w_0,w_1,\dots,w_{r-1}),
\]
where each $w_i$ is interpreted as an offset within layer $i$ (with $w_0=0$ by convention). This parameter is used as a \emph{manual sub-cubing} mechanism: different choices of $\mathbf{w}$ produce independent subinstances.

\paragraph{Geometric meaning.}
Let $L_i$ denote layer $i$ and let $s_i=\sum_{j<i} h_j$ be its starting index, so $L_i=\{s_i,\dots,s_i+h_i-1\}$.
For each $i\ge 1$ we fix an \emph{anchor vertex} $a=s_{i-1}$ in the previous layer $L_{i-1}$, and we select two vertices of the current layer:
\[
b=s_i,\qquad c=s_i+w_i.
\]
We then add unit constraints on orientation literals that force all points of index $>s_i$ (i.e., all remaining vertices of $L_i$ and all deeper layers / leftover points) to lie inside the wedge at $a$ bounded by the rays $ab$ and $ac$.
Equivalently, from the viewpoint of $a$, the pair $(b,c)$ is forced to behave like two extremal (supporting) vertices of the inner layer, fixing part of the relative ``rotation'' between $L_{i-1}$ and $L_i$.
When $w_i=0$ (so $c=b$), this extra anchoring at layer $i$ is omitted.

\paragraph{Role in computation.}
The $\mathbf{w}$ constraints are not intended as a balanced or exhaustive splitting strategy; rather, they provide a small, geometry-informed way to partition some highly symmetric templates into subinstances that can behave very differently in runtime.

\section{Instance size for $(n,k)=(33,7)$}

\subsection{Variable counts}

\paragraph{Triple variables.}
There are $\binom{33}{3}=5456$ unordered triples, hence 5456 base orientation variables.

\paragraph{4-set selectors.}
There are $\binom{33}{4}=40920$ 4-sets and 14 selectors per 4-set, contributing $40920\times 14=572{,}880$ variables.

\paragraph{Total.}
Thus the base instance has
\[
5456 + 572{,}880 = 578{,}336
\]
Boolean variables.

\subsection{Clause counts}

The dominant clause blocks are:
\begin{itemize}
\item \textbf{Reduced 5-point constraints:} for $n=33$ this block contributes $9{,}493{,}440$ clauses in the recorded generator version.
\item \textbf{4-set consistency:} $71\binom{33}{4} = 2{,}905{,}320$ clauses.
\item \textbf{No convex 7-set constraints (via 4-sets):} $\binom{33}{7}=4{,}272{,}048$ clauses, each of length 280.
\end{itemize}
Therefore the base CNF (before hull constraints) has
\[
9{,}493{,}440 \;+\; 2{,}905{,}320 \;+\; 4{,}272{,}048
\;=\;
16{,}670{,}808
\]
clauses. Hull constraints add only unit clauses (hundreds to a few thousands), so they do not change CNF size materially, though they can strongly affect runtime.

\section{Results and observations}\label{sec:results}

This section summarizes a set of observed UNSAT runtimes for $(n,k)=(33,7)$ under convex-layer templates and (optionally) the $\mathbf{w}$ sub-cubing parameter. Times are single-thread wall-clock runtimes and should be interpreted qualitatively.

All CNF instances were generated using Python scripts (PySAT~\cite{pysat_library}) and solved using Kissat~\cite{kissat} (single-thread runs).

\begin{table}[h]
\centering
\begin{tabular}{@{}lll@{}}
\toprule
Layer sizes $\mathbf{h}$ (sum $H$) & sub-cube parameter $\mathbf{w}$ & observed time (s) \\ \midrule
$6,6,6,3,6,6$ ($H=33$) & (not recorded) & $\approx 1.73\times 10^{5}$ \\
$3^{11}$ ($H=33$) & (not recorded) & $\approx 6.05\times 10^{5}$ \\
$3,3,4,3,3,6,6,5$ ($H=33$) & $[0,1,1,1,1,1,4,4]$ & $2.26\times 10^{5}$ \\
\midrule
$5^6$ ($H=30$) & (baseline run) & $2.59\times 10^{5}$ \\
$5^6$ ($H=30$) & $[0,4,4,4,4,4]$ & $2.50\times 10^{3}$ \\
$5^6$ ($H=30$) & $[0,1,1,1,1,1]$ & $1.60\times 10^{4}$ \\
$5^6$ ($H=30$) & $[0,1,1,1,0,0]$ & $2.28\times 10^{6}$ \\
\midrule
$4^8$ ($H=32$) & $[0,3,3,3,3,3,3,3]$ & $8.76\times 10^{4}$ \\
$4^8$ ($H=32$) & $[0,1,2,3,1,2,3,1]$ & $3.19\times 10^{5}$ \\
$4,3,4,4,4,4,4,4$ ($H=31$) & $[0,1,3,2,1,2,1,1]$ & $2.10\times 10^{6}$ \\
$4,3,3,5,4,6,3,4$ ($H=32$) & $[0,2,2,4,1,1,1,1]$ & $3.15\times 10^{3}$ \\
\midrule
$3,5,3,5,3,3,5,5$ ($H=32$) & $[0,1,2,1,2,1,1,4]$ & $3.59\times 10^{4}$ \\
$3,4,3,4,4,3,4,4,3$ ($H=32$) & $[0,1,2,1,2,1,1,3,2]$ & $9.06\times 10^{5}$ \\
\bottomrule
\end{tabular}
\caption{Selected UNSAT runtimes for $(n,k)=(33,7)$ under convex-layer templates and the $\mathbf{w}$ parameter (single-thread wall-clock seconds).}
\label{tab:runtimes}
\end{table}

\begin{remark}[Runtime variability]
Even within closely related convex-layer families, runtimes vary widely. For example, within the $5^6$ template we observe sub-cubes ranging from $2.50\times 10^{3}$\,s to $2.28\times 10^{6}$\,s, and within $4$-dominated templates from $8.76\times 10^{4}$\,s to $2.10\times 10^{6}$\,s. This motivates both encoding refinements (to reduce clause/literal volume and strengthen propagation) and stronger geometric anchoring for the hardest symmetric configurations.
\end{remark}

\paragraph{Reproducibility.}
The SAT instance generator and scripts used to produce the results reported in this paper are publicly available.\footnote{\url{https://github.com/bogdan27182/esc-paper}}
The repository contains the implementation corresponding to the encoding described here, together with configurations for the reported convex-layer templates.

\section{Discussion and future work}

The encoding presented here makes it feasible to generate and solve large SAT instances arising from the $(33,7)$ case, at least for many structured subfamilies obtained by convex-layer anchoring. The main practical challenge is that solver time varies widely across nearby anchored families.

\subsection*{Future work}
We list two encoding directions that appear promising and will be tested experimentally in a future revision. 

\begin{enumerate}
\item \textbf{Reduce the granularity of the 4-set encoding.}
The current implementation distinguishes all realizable 4-point types using multiple selector variables per 4-set. However, the 7-set exclusion constraints ultimately depend only on the coarse distinction ``convex vs.\ non-convex'' for each 4-set. It may therefore be possible to replace the detailed 4-type bookkeeping by a coarser representation that captures only what is needed for the 7-set constraints, reducing auxiliary variables and 4-set clauses.

\item \textbf{Shorten the 7-set constraints using the coarser 4-set information.}
At present, each 7-set clause is a long disjunction expanded over many pattern-specific literals. If a coarser convex/non-convex indicator is available per 4-set, the corresponding 7-set constraints can be expressed with substantially shorter clauses. Besides reducing memory pressure, shorter clauses typically strengthen propagation because a 7-set becomes constrained once many of its 4-subsets are forced convex.
\end{enumerate}

A second (orthogonal) direction is stronger geometric anchoring on the outer convex layer to reduce symmetry while preserving completeness of the intended case split; this will be explored experimentally.

\section{Conclusion}

We presented a SAT encoding for the planar Erd\H{o}s--Szekeres problem at $(n,k)=(33,7)$ using triple-orientation variables, a reduced CC-style 5-point constraint family, and an exact 4-set criterion (Proposition~\ref{prop:4criterion}) to exclude convex 7-sets. The encoding supports additional convex-layer anchoring constraints and yields UNSAT certificates for several anchored families. Further progress appears to depend on both encoding refinements and stronger geometric anchoring for the hardest symmetric configurations.

\end{document}